\documentclass{amsart}
\usepackage[utf8]{inputenc}
\usepackage[colorlinks=true, allcolors=blue]{hyperref}
\usepackage[a4paper,top=3cm,bottom=2cm,left=3cm,right=3cm,marginparwidth=1.75cm]{geometry}

\title{Smallest nontrivial quotients of the commutator subgroup of braid groups}
\author{Sudipta Kolay}
\address{ICERM, 121 South Main Street, 11th Floor\\Providence, RI 02903, USA.}
\email{sudipta\_kolay@brown.edu}
\date{}

\begin{document}
\begin{abstract}
    We prove that the smallest non-trivial quotients of the commutator subgroups of the braid groups are the alternating groups for $n\geq 5$, proving a conjecture of Chudnovsky--Kordek--Li--Partin. Furthermore, we show that any minimal quotient map is the standard projection, composed with an automorphism of the alternating group. 
\end{abstract}

\maketitle
\newtheorem{theorem}{Theorem}
\newtheorem{lemma}[theorem]{Lemma}
\newtheorem{claim}[theorem]{Claim}

\theoremstyle{definition}
\newtheorem{definition}[theorem]{Definition}
\newtheorem{example}[theorem]{Example}
\newtheorem{xca}[theorem]{Exercise}

\theoremstyle{remark}
\newtheorem{remark}[theorem]{Remark}

\section{Introduction}

We show that for $n\geq 5$, the smallest nontrivial quotient of the commutator subgroup $B'_n$ of the Artin braid group $B_n$ is the alternating group $A_n$, proving a conjecture of Chudnovsky--Kordek--Li--Partin \cite[Question 2]{chudnovsky2020finite}, and moreover we characterize all minimal quotient maps.
\begin{theorem}\label{A}
Suppose $n\geq 5$. If $G$ is a nontrivial quotient of $B'_n$, then either $|G|> |A_n|=\frac{n!}{2}$, or $G$ is isomorphic to $A_n$. Moreover, in the latter case the quotient map $B'_n\rightarrow G$ is obtained by postcomposing the natural map $\pi$ with an automorphism\footnote{For $n\geq 5$ and $n\neq 6$, it is known that every automorphism of $A_n$ is given by a conjugation in $S_n$, and for $n=6$, there is exactly one outer automorphism up to conjugation ($n=2,3$ are also exceptional cases).} of $A_n$.
\end{theorem}
Let us note that the corresponding statement is false for $n=4$, since $A_4$ is not simple.
 Chudnovsky--Kordek--Li--Partin~\cite{chudnovsky2020finite} proved lower bounds\footnote{These lower bounds have subsequently been improved by Caplinger--Kordek \cite{caplinger2020small} and Scherich--Verberne \cite{scherich2020finite}.} for the size of non-cyclic quotient of $B_n$, providing evidence in support of Margalit's conjecture \cite[Question 1]{chudnovsky2020finite}: for $n\geq 5$ the smallest non-cyclic quotient of braid group $B_n$ is the symmetric group $S_n$. Also, they made an analogous\footnote{Let us note that the standard forgetful projection map $\pi:B_n\rightarrow S_n$ restricts to give an epimorphism of the commutator subgroups $\pi:B'_n\rightarrow A_n$.} conjecture regarding the commutator subgroups, that for $n\geq 5$ the smallest nontrivial quotients of $B'_n$ is $A_n$, and proved similar lower bounds. Caplinger--Kordek proved~\cite[Theorem 1]{caplinger2020small} the latter conjecture for $n\in\{5,6,7,8\}$ (and Margalit's conjecture for $n\in\{5,6\}$). We will give an inductive proof of Theorem~\ref{A} where we will use this result of Caplinger--Kordek as base cases.

The approach of Caplinger--Kordek in \cite{caplinger2020small} is similar to that of Kielak--Pierro \cite{Kielak2017OnTS} and Baumeister--Kielak--Pierro \cite{Baumeister2019OnTS} in their proof of analogous conjectures in the setting of mapping class groups and outer automorphism groups of free groups due to Zimmermann \cite{Zimmermann2008ANO} and Mecchia--Zimmermann~\cite{Mecchia2009OnMF}, respectively. This approach crucially relies on the classification of finite simple groups, and tries to rule out quotients of smaller order.  The author gave  elementary proofs \cite{Kolay2021BMCG,Kolay2021AFn} of the aforementioned conjectures of Margalit, Zimmermann, and Mecchia--Zimmermann using the inductive orbit stabilizer method \cite[Section 3]{Kolay2021BMCG}, and the present paper shows the same techniques answer the conjecture by Chudnovsky--Kordek--Li--Partin. In this context, we will prove (most of) Theorem~\ref{A} by carrying out in Sections \ref{lb} and \ref{ind} respectively, the following two steps of the inductive orbit stabilizer method.
\begin{enumerate}
    \item For $n\geq 8$, the conjugacy class of the image of $\sigma_2\sigma_1^{-1}$ under any nontrivial quotient of $B'_n$ has at least $2{ n\choose 3}$ elements.
    \item Inductively, the stabilizer of the image of $\sigma_2\sigma_1^{-1}$ has cardinality at least $\frac{3}{2}(n-3)! $, and the result follows by the orbit stabilizer theorem.
\end{enumerate}

 These two steps will prove Theorem~\ref{A} only for the case $n \geq 8$. The remaining cases $n\in\{ 5,6,7\}$ is computationally verified by Noah Caplinger, adapting previous work in \cite{caplinger2020small}, using very similar code~\cite{code}. However, it may be possible to prove Theorem~\ref{A} without any computer assistance, for instance by improving the $n\geq 8$ hypothesis in the lower bounding orbit size step by a more careful case analysis, and using similar ideas as in \cite[Proof of Theorem 1 for $n=6$]{Kolay2021BMCG}.

Along the way to prove our main result, we show in the next section that for $n\geq 5$, any one of the elements  $\sigma_2\sigma_1^{-1}$, or $(\sigma_2\sigma_1^{-1})^2$, or $\sigma_3\sigma_1^{-1}$ normally generate the commutator subgroup $B'_n$.\\

\noindent\emph{Background}: We refer the reader to  \cite[Chapter 1]{lin2004braids} for a comprehensive discussion on the topics of braid groups and their commutator subgroups, and \cite[Section 2]{Kolay2021BMCG} for a brief introduction. We will introduce/cite most notation needed along the way. \\

\noindent\emph{Acknowledgments}: The author would like to thank Noah Caplinger for the computer verification of the main theorem in the exceptional cases $n\in\{5,6,7\}$. The author is grateful to Dan Margalit for pointing out the shorter proof of Lemma~\ref{lm2} and various discussions. The author thanks John Etnyre for helpful suggestions. This work is supported by NSF grants DMS-1439786 and DMS-1906414 while the author is/was located at ICERM and Georgia Tech respectively.\\

\section{Normal generators of the commutator subgroup of the braid group}

In this section we will prove some facts about normal generators of the commutator subgroup $B'_n$, which may be of independent interest. It is known \cite[Remark 1.10]{lin2004braids} that for $n\geq 5$ the normal closure in $B_n$ of either the element $\sigma_2\sigma_1^{-1}$ or  $\sigma_3\sigma_1^{-1}$ is the commutator subgroup $B'_n$. Here we slightly restate this result by just looking at the normal closure only in $B'_n$.

\begin{lemma}\label{lemA}
For $n\geq 5$, the normal closure (in $B'_n$) of either the element $\sigma_2\sigma_1^{-1}$ or  $\sigma_3\sigma_1^{-1}$ is $B_n'$.
\end{lemma}
\begin{proof}
We use the well known (but very useful) fact that the commutator subgroup precisely consists of all braids which have exponent sum 0. This implies that if we show that two elements of $B'_n$ are conjugate by some element in $B_n$, and we can find a half twist commuting with one of the elements, then the two elements are conjugate by an element of $B'_{n}$ (we can multiply the conjugating braid with suitable power of the commuting half twist so as to get exponent sum 0), and we will repeatedly use this observation implicitly hereafter. Let us note that in $B_n$ with $n\geq 5$, the braids $\sigma_2\sigma_1^{-1}$ or  $\sigma_3\sigma_1^{-1}$ commute with the half twist $\sigma_4$ and $\sigma_1$ respectively, and so the lemma follows by the aforementioned fact \cite[Remark 1.10]{lin2004braids}.\end{proof}

Now we will look at a strengthening of the above lemma, for which we will give two proofs.

\begin{lemma}\label{lm2}
For $n\geq 5$, the normal closure (in $B'_n$) of the element $(\sigma_2\sigma_1^{-1})^2$ is $B_n'$. 
\end{lemma}

\begin{proof} We will use the notation of \cite[Remark 1.8]{lin2004braids} of the Gorin-Lin presentation (see also  \cite{Gorin_1969}) of the commutator subgroup $B_n'$ for $n\geq 5$. 
That is, we have $$u=\sigma_2\sigma_1^{-1},\quad v=\sigma_1\sigma_2\sigma_1^{-2},\quad w=\sigma_2\sigma_3\sigma_1^{-1}\sigma_2^{-1},\quad c_1=\sigma_3\sigma_1^{-1}.$$ 
Let us denote the quotient class of an element $x$ in $B'_n$ by $\overline{x}$ after modding out by the relation $(\sigma_2\sigma_1^{-1})^2=1$. In the quotient, we 
have $\overline{\sigma_2\sigma_1^{-1}}=\overline{\sigma_1\sigma_2^{-1}}$, and thus by a conjugation we see that $\overline{\sigma_3\sigma_2^{-1}}=\overline{\sigma_2\sigma_3^{-1}}$. Therefore we obtain 
\begin{equation}\label{aa}
 \overline{c_1^{-1}}=\overline{\sigma_1\sigma_3^{-1}} =\overline{\sigma_1\sigma_2^{-1}\sigma_2\sigma_3^{-1}} =\overline{\sigma_2\sigma_1^{-1}\sigma_3\sigma_2^{-1}}=\overline{w}   
\end{equation}

We note that since $u$ and $v$ are conjugate, they must both have the same order two in the quotient, i.e. $\overline{u^{-1}}=\overline{u}$ and $\overline{v^{-1}}=\overline{v}$.
By \cite[Equation 1.16]{lin2004braids} and \cite[Equation 1.17]{lin2004braids} we see that $\overline{c_1}=\overline{uwu}=\overline{w^2c_1^{-1}w}$; and hence together with \eqref{aa}, this implies $\overline{w^{-1}}=\overline{w^{4}}$, i.e. $\overline{w^{5}}=1$.
Similarly, we note that \cite[Equation 1.18]{lin2004braids} and \cite[Equation 1.19]{lin2004braids} can be rephrased in the present context as $\overline{vc_1v}=\overline{w^2}$ and $\overline{vwv}=\overline{w^9}=\overline{w^{-1}}$, respectively. As $\overline{vc_1v}$ and $\overline{vwv}$  are inverses of each other it follows that $\overline{w^2}$ must be the inverse of $\overline{w^{-1}}$, i.e. $\overline{w^2}=\overline{w}$, or in other words $\overline{w}=1=\overline{c_1}$. The result now follows from Lemma~\ref{lemA}.\end{proof}

Dan Margalit pointed out to us that Lemma~\ref{lm2} can also be proved more succinctly using the well suited arc criterion~\cite[Lemma 4.2]{chen2019homomorphisms}, as explained below after introducing some terminology.

\noindent\emph{Notation}: Suppose $\gamma_{i,j}$ (respectively $\delta_{i,j}$) denotes the arc  joining the $i$th and $j$th punctures going above (respectively below) the intermediate punctures. Also, we let $\rho_{ij}$ (respectively $\varrho_{ij}$) denote the right handed half twist about the arc $\gamma_{i,j}$ (respectively $\delta_{i,j}$). We note that $\rho_{ij}$'s are the Birman-Ko-Lee generators \cite{BIRMAN1998322} of $B_n$, and moreover $\rho_{ij}$ equals the Artin generator $\sigma_i$ for $j=i+1$.
\begin{proof}[Alternate Proof of Lemma~\ref{lm2}] Let $\gamma_\star:=(\sigma_2\sigma_1^{-1})^2(\gamma_{24})$, and let $\rho_\star$ denote the right handed half twist about $\gamma_\star$. We note that $f:=\rho_\star\rho_{24}^{-1}=(\sigma_2\sigma_1^{-1})^2 \rho_{24} (\sigma_2\sigma_1^{-1})^{-2} \rho_{24}^{-1} $ is in the normal closure\footnote{Since $(\sigma_2\sigma_1^{-1})^2$ commutes with $\sigma_4$, we see that normal closures in $B_n$ and $B'_n$ coincide.} of $(\sigma_2\sigma_1^{-1})^2$. The result now follows by setting $c:=\delta_{3,5}$ and applying the well suited arc criterion~\cite[Lemma 4.2]{chen2019homomorphisms} since $c$ and $f(c)$ share exactly one  endpoint and have disjoint interiors.\end{proof}

\section{Lower bound on orbit size}\label{lb}
We begin by observing that the conjugacy class of all three cycles in $A_n$ consists of $2{ n\choose 3}$ elements, and we now find the same number of conjugate elements in $B'_n$, and show that they all must map to distinct elements under any nontrivial quotient.\\

\emph{The orbit}: Consider for any triple $i,j,k$ (with $1\leq i<j<k\leq n$) the two elements $\alpha_{ijk}=\rho_{jk}\rho_{ij}^{-1}$ and $\beta_{ijk}=\rho_{ij}^{-1}\rho_{jk}$. We begin by observing that all these $2{ n\choose 3}$ elements are in the same orbit under the conjugation action in $B'_n$ by a change of coordinates. Furthermore, let us note that Lemma~\ref{lemA} implies that for $n\geq 5$, each $\alpha_{ijk}$ and $\beta_{ijk}$ is a normal generator of $B'_n$, and hence must map to nontrivial elements under any nontrivial quotient. The following lemma shows that moreover, the images of these elements must be pairwise distinct.

\begin{lemma}\label{lemB}
For $n\geq 8$, the $2{ n\choose 3}$ elements $\{\alpha_{ijk}, \beta_{ijk}:1\leq i<j<k\leq n\}$ must map to distinct elements under any non-trivial quotient $q:B'_n\rightarrow G$.
\end{lemma}
\noindent\emph{Notation}: For any $x\in B'_n$, let us denote the quotient class $q(x)$ in $G$ by $\overline{x}$. 
\begin{proof}
We will prove the lemma by the method of  contradiction. So let us assume that $\overline{\theta_{ijk}}=\overline{\psi_{lmp}}$ (where $\theta$ or $\psi$ can be either the symbol $\alpha$ or $\beta$).
If $\{i,j,k\}=\{l,m,p\}$, then we are in the case $\overline{\alpha_{ijk}}=\overline{\beta_{ijk}}$, then we get $\overline{\rho_{jk}\rho_{ij}^{-1}}=\overline{\rho_{ij}^{-1}\rho_{jk}}$, and hence $\overline{\rho_{ij}\rho_{jk}\rho_{ij}^{-1}\rho_{jk}^{-1}}=1$. We see that $\rho_{jk}\rho_{ij}^{-1}\rho_{jk}^{-1}=\rho_{ik}^{-1}$, and thus $\rho_{ij}(\rho_{jk}\rho_{ij}^{-1}\rho_{jk}^{-1})$ is conjugate to $\sigma_2\sigma_1^{-1}$, it follows from Lemma~\ref{lemA} that $G$ must be trivial, a contradiction.\\
   In case $\{i,j,k\}\neq\{l,m,p\}$, let us assume that $k$ is an end vertex for exactly one supporting arc of $\theta_{ijk}$ (a symmetric argument works in other cases). Since  $n\geq 8$, we see that we can find indices $r,s$ not contained in $\{i,j,k,l,m,p\}$.
Now we see that $\overline{\varrho_{k,r}\varrho_{r,s}^{-1}}$ commutes with $\overline{\psi_{lmp}}=\overline{\theta_{ijk}}$. By a change of coordinates (conjugation in $B'_n$) we may assume without loss of generality that $\overline{\sigma_2\sigma_1^{-1}}$ commutes with $\overline{\sigma_4\sigma_3^{-1}}$. 
In other words, we have $\overline{\sigma_2\sigma_1^{-1}\sigma_4\sigma_3^{-1}}=\overline{\sigma_4\sigma_3^{-1}\sigma_2\sigma_1^{-1}}$, and consequently we obtain by the far commutation relation: $\overline{\sigma_2\sigma_3^{-1}}=\overline{\sigma_3^{-1}\sigma_2}$, or equivalently $1=\overline{(\sigma_3\sigma_2\sigma_3^{-1})\sigma_2^{-1}}=\overline{\rho_{2,4}\sigma_2^{-1}}$.
The result now follows by a change of coordinates   and Lemma~\ref{lemA}.\end{proof}

We can say a bit more in case there are exactly $2{ n\choose 3}$ elements in the orbit of $\overline{\sigma_2\sigma_1^{-1}}$ in $G$.

\begin{lemma}\label{lemC}
For $n\geq 8$, if the conjugacy class of $\overline{\sigma_2\sigma_1^{-1}}$ under a nontrivial quotient of $B'_n$ consists of exactly $2{ n\choose 3}$ elements, then we must have $(\overline{\sigma_i\sigma_1^{-1}})^2=1$ for any $3\leq i\leq n-1$, and $\overline{\sigma_2\sigma_1^{-1}}=\overline{\sigma_2^{-1}\sigma_1}$.
\end{lemma}
\begin{proof}
Let us consider the four elements $\overline{\sigma_1\sigma_2^{-1}},\overline{\sigma_1^{-1}\sigma_2},\overline{\sigma_2\sigma_1^{-1}},\overline{\sigma_2^{-1}\sigma_1}$. If these four quotient classes are pairwise distinct, the second half of the proof of Lemma~\ref{lemB} in fact shows that the orbit of $\overline{\sigma_2\sigma_1^{-1}}$ must consist of at least $4{ n\choose 3}$ elements. Also, we note the only way these do not give four distinct quotient classes under a nontrivial quotient is if $\overline{\sigma_2\sigma_1^{-1}}=\overline{\sigma_2^{-1}\sigma_1}$ and (by symmetry) $\overline{\sigma_1\sigma_2^{-1}}=\overline{\sigma_1^{-1}\sigma_2}$  (the other possibilities are ruled out by Lemmas~\ref{lemA} and \ref{lm2}). If this is the case, then by conjugation we also get $\overline{\sigma_3\sigma_2^{-1}}=\overline{\sigma_3^{-1}\sigma_2}$  and $\overline{\sigma_2\sigma_3^{-1}}=\overline{\sigma_2^{-1}\sigma_3}$. Consequently, we have:
\begin{equation}
    \overline{\sigma_3\sigma_1^{-1}}=\overline{\sigma_3\sigma_2^{-1}}\overline{\sigma_2\sigma_1^{-1}}
    =\overline{\sigma_3^{-1}\sigma_2}\overline{\sigma_2^{-1}\sigma_1}=\overline{\sigma_3^{-1}\sigma_1}
\end{equation}
As $\sigma_3\sigma_1^{-1}=\sigma_1^{-1}\sigma_3$, we see that the quotient class of $\sigma_3\sigma_1^{-1}$ must be an involution, and the rest follows since $\sigma_i\sigma_1^{-1}$ is conjugate to $\sigma_3\sigma_1^{-1}$ for any $3\leq i\leq n-1$.
\end{proof}

\section{Inductive step}\label{ind}
In this final section we will complete the proof of Theorem~\ref{A} for $n\geq 8$, by using a version of induction, in steps of three, with two families of hypothesis.\\
\emph{Two families of hypotheses}: For any natural number $n$, let $P(n)$ and $Q(n)$ be the statements:\\
\noindent $P(n)$: If $G$ is a nontrivial quotient of $B'_n$, then $|G|\geq \frac{n!}{2}$.\\
\noindent $Q(n)$: If $G$ is a nontrivial quotient of $B'_n$, then either $|G|> |A_n|=\frac{n!}{2}$, or $G$ is isomorphic to $A_n$. Moreover, in the latter case the quotient map $B'_n\rightarrow G$ is obtained by postcomposing the natural map $\pi$ with an automorphism of $A_n$.\\

\begin{proof}[Proof of Theorem~\ref{A}]
\noindent If $k\geq 5$, and the weaker hypothesis $P(k)$ is true, then we will show that so is the stronger hypothesis $Q(k+3)$, thereby proving Theorem~\ref{A} for $n\geq 8$.

\noindent\emph{Base Cases}: $P(5),P(6)$ and $P(7)$ are true by the work of Caplinger--Kordek~\cite[Theorem 1]{caplinger2020small}.
As mentioned earlier, $Q(5),Q(6)$ and $Q(7)$ are verified by Caplinger~\cite{code} with computer assistance.\\

\noindent\emph{Inductive step}: Suppose $n\geq 8$, and $P(n-3)$ is true. We would like to show that $Q(n)$ holds as well. We use the same notation as above to let $\overline{x}$ denote the quotient class under the map $q:B'_n\rightarrow G$.
As $n\geq 8$, we see by Lemma~\ref{lemB} that the conjugacy class of $\overline{\sigma_2\sigma_1^{-1}}$ in $G$ has at least $2{ n\choose 3}$ elements. By disjointedness, we that $\sigma_2\sigma_1^{-1}$  commutes with the subgroup $R\cong B'_{n-3}$ consisting of the commutator subgroup of the braid group coming from the rightmost $n-3$ punctures. Thus the centralizer of $\overline{\sigma_2\sigma_1^{-1}}$ contains both the cyclic subgroup $M$ of $G$ generated by $\overline{\sigma_2\sigma_1^{-1}}$, and $\overline{R}$, the image of $R$ under the quotient map. Let $Z=M\cap\overline{R}$. We note that $Z$ is central in $\overline{R}$, and therefore $\overline{R}/Z$ must be a nontrivial quotient of $ B'_{n-3}\cong R$ (in case it was trivial, then $\overline{R}=Z$ is a
abelian, whence we get a nontrivial abelian quotient of $ B'_{n-3}\cong R$, contradicting that $B'_{n-3}$ is perfect). By the induction hypothesis $P(n-3)$, it follows that $|\overline{R}/Z|\geq \frac{(n-3)!}{2}$.
As the centralizer $C$ of $\overline{\sigma_2\sigma_1^{-1}}$ contains the subgroup generated by $M$ and $\overline{R}$, it follows that $|C|\geq |M||\overline{R}/Z|\geq |M|\frac{(n-3)!}{2}$. Combining with the lower bound of the orbit size from Lemma~\ref{lemB}, we see by the orbit stabilizer theorem that:
\begin{equation}
    |G|\geq 2{ n\choose 3}|M|\frac{(n-3)!}{2}=\frac{(n!)|M|}{6}
\end{equation}
From Lemma~\ref{lm2} we see that $|M|\geq 3$ (as $\overline{\sigma_2\sigma_1^{-1}}$ cannot have order 1 or 2 in a nontrivial quotient), and thus $|G|\geq \frac{n!}{2} $, i.e. $P(n)$ is true. Moreover in case of equality above, we see that $|M|=3$ and the conjugacy class of $\overline{\sigma_2\sigma_1^{-1}}$ has exactly $2{ n\choose 3}$ elements. Therefore it follows that $\overline{\sigma_2\sigma_1^{-1}}$ has order 3, and by Lemma~\ref{lemC} we also see that $\overline{\sigma_i\sigma_1^{-1}}$ has order 2 for all $3\leq i\leq n-1$.
Now it follows from the Carmichael presentation \cite[page 172]{CarmichaelIntroductionTT} of the alternating group that $q:B'_n\rightarrow G$ factors through $A_n$. Since the alternating groups $A_n$ are simple for $n\geq 5$, we see that the map from $A_n$ to $G$ must be an isomorphism, and so $Q(n)$ is true and the proof is complete.\end{proof}

\bibliographystyle{plain}
\bibliography{references}

\end{document}